\documentclass[11pt,a4paper]{article}

\usepackage{amsmath}
\usepackage{amssymb,verbatim}
\usepackage{amsfonts}
\usepackage{amsthm}
\usepackage{pifont}
\usepackage{graphicx}

\newtheorem{theorem}{Theorem}

\newtheorem{corollary}{Corollary}

\newcommand{\beq}{\begin{equation*}}
\newcommand{\eeq}{\end{equation*}}
\newcommand{\beqn}{\begin{equation}}
\newcommand{\eeqn}{\end{equation}}
\newcommand{\ba}{\begin{array}}
\newcommand{\ea}{\end{array}}
\newcommand{\ds}{\displaystyle}

\newcommand{\RR}{\mathbb R}
\newcommand{\ZZ}{\mathbb Z}
\newcommand{\dd}{\mathrm{d}}
\newcommand{\R}{\mathcal{R}_{a,\,b,\,c}}
\newcommand{\Rha}{\mathcal{R}_{h_a}}
\newcommand{\Rhb}{\mathcal{R}_{h_b}}
\newcommand{\Rhc}{\mathcal{R}_{h_c}}
\newcommand{\C}{\mathcal{C}}
\newcommand{\Q}{\mathcal{Q}}
\newcommand{\T}{\mathcal{T}}
\newcommand{\E}{\mathrm{E}(Z)}

\begin{document}

\title{Hitting probabilities for random convex bodies and lattices of triangles}
\author{Uwe B\"asel}
\date{}
\maketitle

\vspace{-0.6cm}

\begin{abstract}
\noindent In the first part of this paper, we obtain symmetric formulae for the probabilities that a plane convex body hits exactly 1, 2, 3, 4, 5 or 6 triangles of a lattice of congruent triangles in the plane. Furthermore, a very simple formula for the expectation of the number of hit triangles is derived. In the second part, we calculate the hitting probabilities in the cases where the convex body is a rectangle, an ellipse and a half disc. Already known results for a line segment (needle) follow as special cases of the rectangle and the ellipse.\\[0.2cm]
\end{abstract}

\vspace{-0.3cm}

\section{Introduction}
\vspace{-0.15cm}
We consider the random throw of a planar convex body $\C$ onto a plane with an unbounded lattice $\R$ of congruent triangles with sides $a$, $b$, $c$, and opposite angles $\alpha$, $\beta$, $\gamma$, respectively. (An example is shown in Fig.\:\ref{R}.) As usual, we denote the altitudes from the sides $a$, $b$, $c$ respectively as $h_a$, $h_b$, $h_c$, and have
\begin{align*}
  h_a = {} & b\sin\gamma = c\sin\beta\,,\quad 
  h_b = a\sin\gamma = c\sin\alpha\,,\quad
  h_c = a\sin\beta = b\sin\alpha\,.
\end{align*}
$\R$ is the union of three families $\Rha$, $\Rhb$ and $\Rhc$ of parallel lines,
\begin{align*}
  \Rha 
  = {} & \{(x,y)\in\RR^2\mid x\sin\beta+y\cos\beta=kh_a,\,k\in\ZZ\}\,,\\
  \Rhb
  = {} & \{(x,y)\in\RR^2\mid x\sin\alpha-y\cos\alpha=kh_b,\,k\in\ZZ\}\,,\\
  \Rhc
  = {} & \{(x,y)\in\RR^2\mid y=kh_c,\,k\in\ZZ\}\,.
\end{align*}
Every triangle of $\R$ is congruent to the triangle
\beq
  \T:=\{(x,y)\in\mathbb{R}^2\;|\;0\leq y\leq h_c\,,\;
			y\cot\alpha\leq x\leq c-y\cot\beta\}\,.
\eeq
Two triangles form a parallelogram that is congruent to the parallelogram
\beq
  \Q:=\{(x,y)\in\mathbb{R}^2\;|\;0\leq y\leq h_c\,,\;
			y\cot\alpha\leq x\leq c+y\cot\alpha\}\,,
\eeq 
see Fig.\:\ref{R} and \ref{Q}. The area $Q$ of $\Q$ is given by 
\beqn \label{area_Q}
  Q = ab\sin\gamma = ac\sin\beta = bc\sin\alpha\,.
\eeqn
Let $\C$ be provided with a fixed reference point $O$ inside $\C$ and a fixed oriented line segment $\sigma$ starting in $O$ (see Fig.\:\ref{Q}), and let $\phi$ be the angle between a fixed direction in the plane and segment $\sigma$. The {\em random throw of $\C$ onto $\R$} is defined as follows: The coordinates $x$,\;$y$ of $O$ are random variables uniformly distributed in $[y\cot\alpha,h_b\csc\alpha+y\cot\alpha]$ and $[0,h_c]$, respectively; the angle $\phi$ is a random variable uniformly distributed in $[0,2\pi]$. All three random variables are stochastically independent. Let $u$ denote the perimeter of $\C$, $F$ the area of $\C$, and $s:\RR\rightarrow\RR$, $\phi\mapsto s(\phi)$ the support function of $\C$ with reference to the point~$O$. We use $s$ in the following sense (see Fig.\:\ref{s}): $\C$ rotates about the origin of a Cartesian $x,y$-coordinate system with rotation angle $\phi$ between the $x$-axis and segment $\sigma$, whereby $O$ coincides with the origin. The support line, touching $\C$, is perpendicular to the $x$-axis. Now, the support function $s$ is the distance between the origin and the support line. In \cite[pp.\:2-3]{Santalo2}, the support function, say $\tilde{s}$, is defined for fixed $\C$ and moving support line. Clearly, the relation between these two definitions is given by $s(\phi)=\tilde{s}(-\phi)$ if the fixed body $\C$ is in the position with $\phi=0$, and $O$ coinciding with the origin. The width of $\C$ in the direction $\phi$ is given by $w:\RR\rightarrow\RR$, $\phi\mapsto w(\phi)=s(\phi)+s(\phi+\pi)$.

\begin{figure}[h]
  \begin{center}
	\includegraphics[scale=0.75]{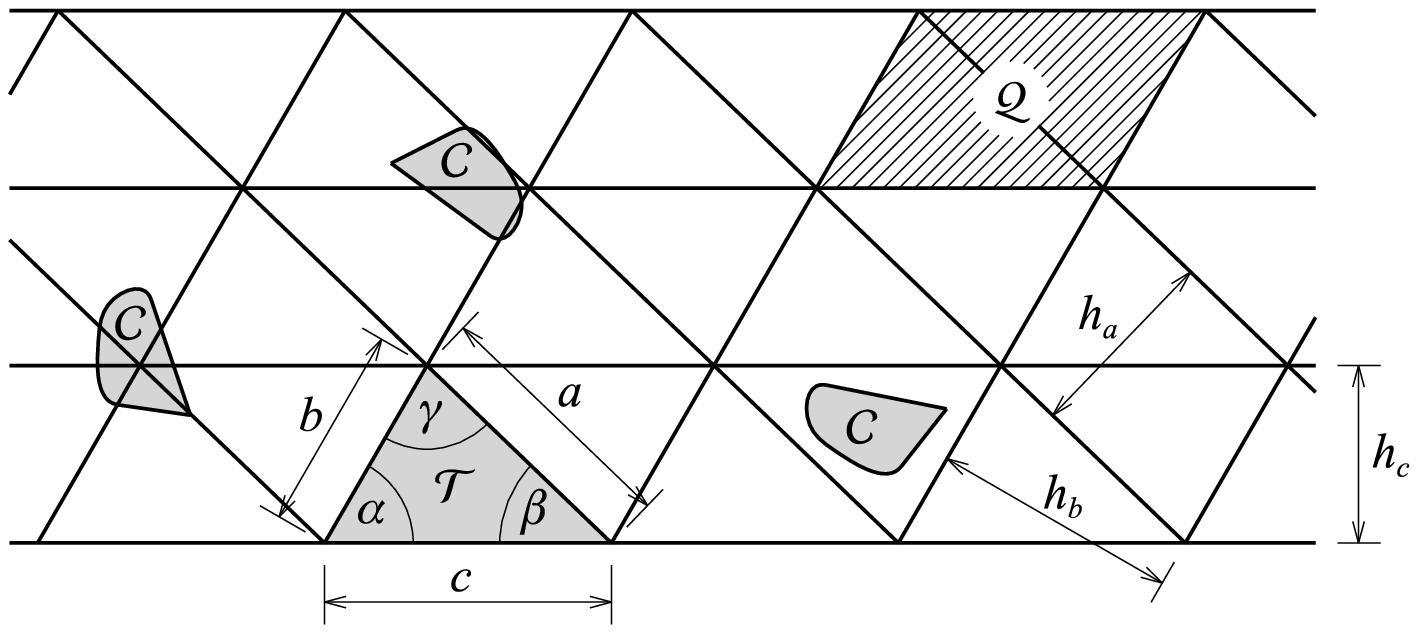}
  \end{center}
  \vspace{-0.4cm}
  \caption{\label{R} Lattice $\R$ and convex body $\C$}
\end{figure}

\noindent 
Markov \cite[pp.\:169-173]{Markov} calculated the probability $p(1)$ that a line segment (needle) of length $\ell$ hits exactly one triangle of $\R$. He found
\beqn \label{p(1)}
  p(1) = 1 + \frac{\ell^2(\alpha a^2+\beta b^2+\gamma c^2)}{2\pi Q^2} 
		- \frac{\ell(4a+4b+4c-3\ell)}{2\pi Q}\,.
\eeqn
Pettineo \cite[pp.\:289-294]{Pettineo} and B\"ottcher \cite[pp.\:3-10]{Boettcher} calculated the probability~$p$ that the needle hits at least one line of $\R$. Clearly, $p=1-p(1)$. Although giving equal results, the formulae of Markov, Pettineo, and B\"ottcher are quite different. Santal\'o \cite[pp.\:167-170]{Santalo1} (see also \cite[pp.\:140-141]{Santalo2}) calculated the probabilities that the needle hits $0$, $1$, $2$ or $3$ lines of the lattice $\mathcal{R}_{a,\,a,\,a}$ of equilateral triangles. Duma and Stoka \cite[pp.\:16-20]{Duma_Stoka2} obtained the probabilities that $\C$ hits at least one line, and exactly three lines of $\R$. Ren and Zhang \cite[pp.\:320-321]{Ren_Zhang} (see also \cite[pp.\:72-73]{Ren}) computed the probabilities for the number of hit lines for a bounded convex body $\C$ with no size restriction, and a lattice of parallelograms. B\"asel and Duma \cite{Baesel_Duma1}, \cite{Baesel_Duma2} calculated the probabilities that $\C$ hits exactly $1$, $2$, $3$ or $4$ parallelograms for two kinds of parallelogram lattices.  

\section{Hitting probabilities}

In the following, we choose $\phi$ as the angle between the direction perpendicular to the sides $b$ and segment $\sigma$ (see Fig.\:\ref{Q}). We assume that for every angle $\phi$, $0\leq\phi<2\pi$, there is a position of the reference point $O$ so that $\C$ with angle $\phi$, denoted as $\C_\phi$, is entirely contained in exactly one triangle $\T$ of $\R$. Therefore, we consider the triangle $\T^*(\phi)$ that is similar to $\T$, and every side of $\T^*(\phi)$ touches $\C_\phi$ (see Fig.\:\ref{T*}). The condition that $\C_\phi$ is entirely contained in $\T$ is equivalent to the fact that $\T^*(\phi)$ is smaller than $\T$. Hence the length of  side $c$ of $\T$ is greater than the length of the respective side $c^*(\phi)$ of $\T^*(\phi)$. Using Fig.\:\ref{T*}, we find
\beq
  c^*(\phi) = s(\phi)\csc\alpha+s(\phi+\alpha+\beta)\csc\beta
				+s(\phi+\alpha+\pi)(\cot\alpha+\cot\beta)\,.
\eeq 
So we have
\beqn \label{max}
  \max_{0\leq\phi<2\pi}c^*(\phi) \leq c\,,
\eeqn
(The equals sign does not influence the calculation of the probabilities.) The radius $\rho$ of the incircle of $\T$ is given by $\rho=Q/(a+b+c)$ \cite[p.\:6]{Goehler}. Clearly, if $\max_{\,0\leq\phi<\pi} w(\phi)\leq 2\rho$ (cf.\:\cite[p.\:16]{Duma_Stoka2}), then condition \eqref{max} always holds true.

\begin{figure}[h] 
\begin{minipage}[h]{5cm}
  \vspace{0.82cm}
  \begin{center}
	\includegraphics[scale=0.85]{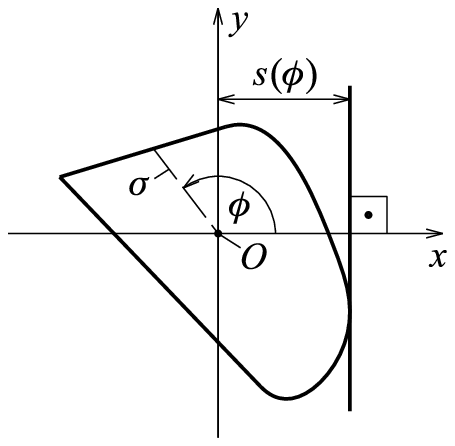}
	\vspace{0.7cm}
	\caption{Support function $s$}
	\label{s}   
  \end{center}
\end{minipage}
\begin{minipage}[h]{7cm}
  \centering
	\includegraphics[scale=0.85]{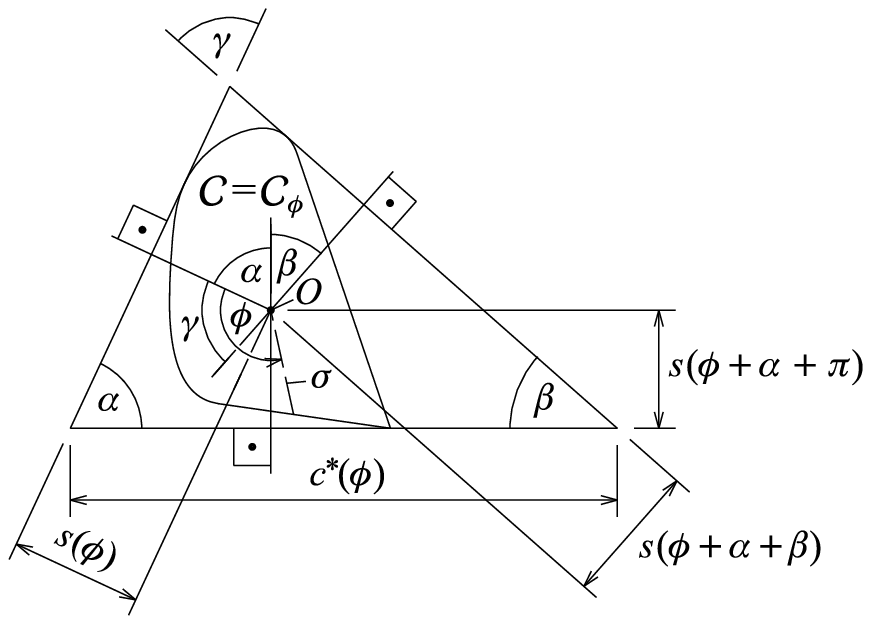}
	\caption{Triangle $\T^*(\phi)$}
	\label{T*}
\end{minipage}
\end{figure}

\begin{theorem} \label{Thm1}
If condition \eqref{max} holds true, the probabilities $p(i)$ that $\C$ hits exactly $i$ triangles of $\R$ are given by
\begin{align*}
  p(1) = {} & 1 - \frac{(a+b+c)u}{\pi Q} + \frac{(a^2+b^2+c^2)J(0)}{2\pi Q^2}
		     + \frac{bcf_1(\alpha)+caf_1(\beta)+abf_1(\gamma)}{\pi Q^2}\,,\\[0.15cm]
  p(2) = {} & \frac{(a+b+c)u}{\pi Q} - \frac{3(a^2+b^2+c^2)J(0)}{2\pi Q^2}
		     - \frac{bcf_2(\alpha)+caf_2(\beta)+abf_2(\gamma)}{\pi Q^2}\,,
				\displaybreak[0]\\[0.15cm]
  p(3) = {} & \frac{3(a^2+b^2+c^2)J(0)}{2\pi Q^2}
		     + \frac{bcf_3(\alpha)+caf_3(\beta)+abf_3(\gamma)}{\pi Q^2}\,,
				\displaybreak[0]\\[0.15cm]
  p(4) = {} & \frac{bcJ(\alpha)+caJ(\beta)+abJ(\gamma)}{\pi Q^2}
		     - \frac{(a^2+b^2+c^2)J(0)}{2\pi Q^2} - \frac{F}{Q}\,,\\[0.15cm]
  p(5) = {} & 0 \,,\quad p(6) = \frac{F}{Q}\,,                      
\end{align*}
with
\begin{align*}
  f_1(x) = {} & I(x)-J(x) \,,\quad f_2(x) = 2I(x)-3J(x) \,,\quad
		       f_3(x) = I(x)-3J(x)\,,\displaybreak[0]\\[0.2cm]
  I(x)\: = {} & \int_0^\pi w(\phi)w(\phi+x)\,\dd\phi \,,\quad
  J(x) = \int_0^{2\pi}s(\phi)s(\phi+x)\,\dd\phi\,,
\end{align*}
and the expectation for the random number $Z$ of hit triangles by
\beq
  \E = 1+\frac{(a+b+c)u}{\pi Q}+\frac{2F}{Q}\,.
\eeq
\end{theorem}

\begin{proof}
It is sufficient to consider only the cases where $O\in\Q$. (It would even be sufficient to consider only the cases where $O\in\T$. But considering $O\in\Q$ is much more convenient.) Since $\C$ is convex, it follows that $p(5)=0$. For fixed value of $\phi$, $\C$ hits exactly $i\in\{1,2,3,4,6\}$ parallelograms of $\R$ if $O$ is inside the set with number $i$ (see Fig.\:\ref{Q}), where every set $i$ is the union of all sets with equal shading. In the following, we use the law of total probability in the form
\beq
  p(i) = \int_0^{2\pi}p(i\,|\,\phi)\;\frac{\dd\phi}{2\pi}\,,
\eeq
where $p(i\,|\,\phi)$ denotes the conditional probability of exactly $i$ hits for fixed value of $\phi$. We have
\beq
  p(i\,|\,\phi) = \frac{F_i(\phi)}{Q}\,,
\eeq
where $F_i(\phi)$ is the area of set $i$. By rearranging the parallelogram $\Q$, one gets the parallelogram shown in Fig.\:\ref{Qnew}. With the help of this figure it is easier to find and calculate the areas $F_i(\phi)$.

\begin{figure}[h]
  \begin{center}
	\includegraphics[scale=0.86]{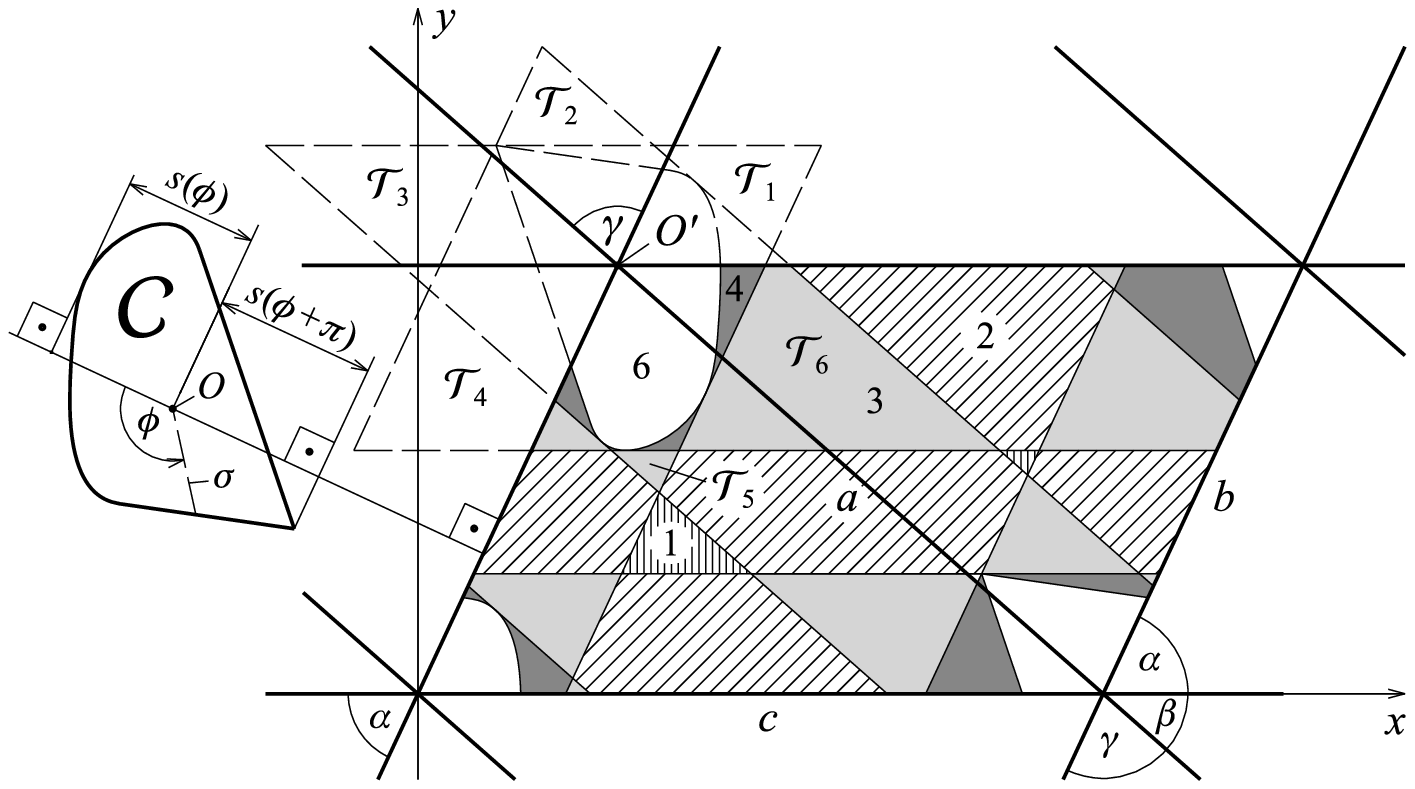}
  \end{center}
  \vspace{-0.4cm}
  \caption{\label{Q} Situation in $\Q$ for fixed value of the angle $\phi$}
  \vspace{1.3cm}
  \begin{center}
	\includegraphics[scale=0.95]{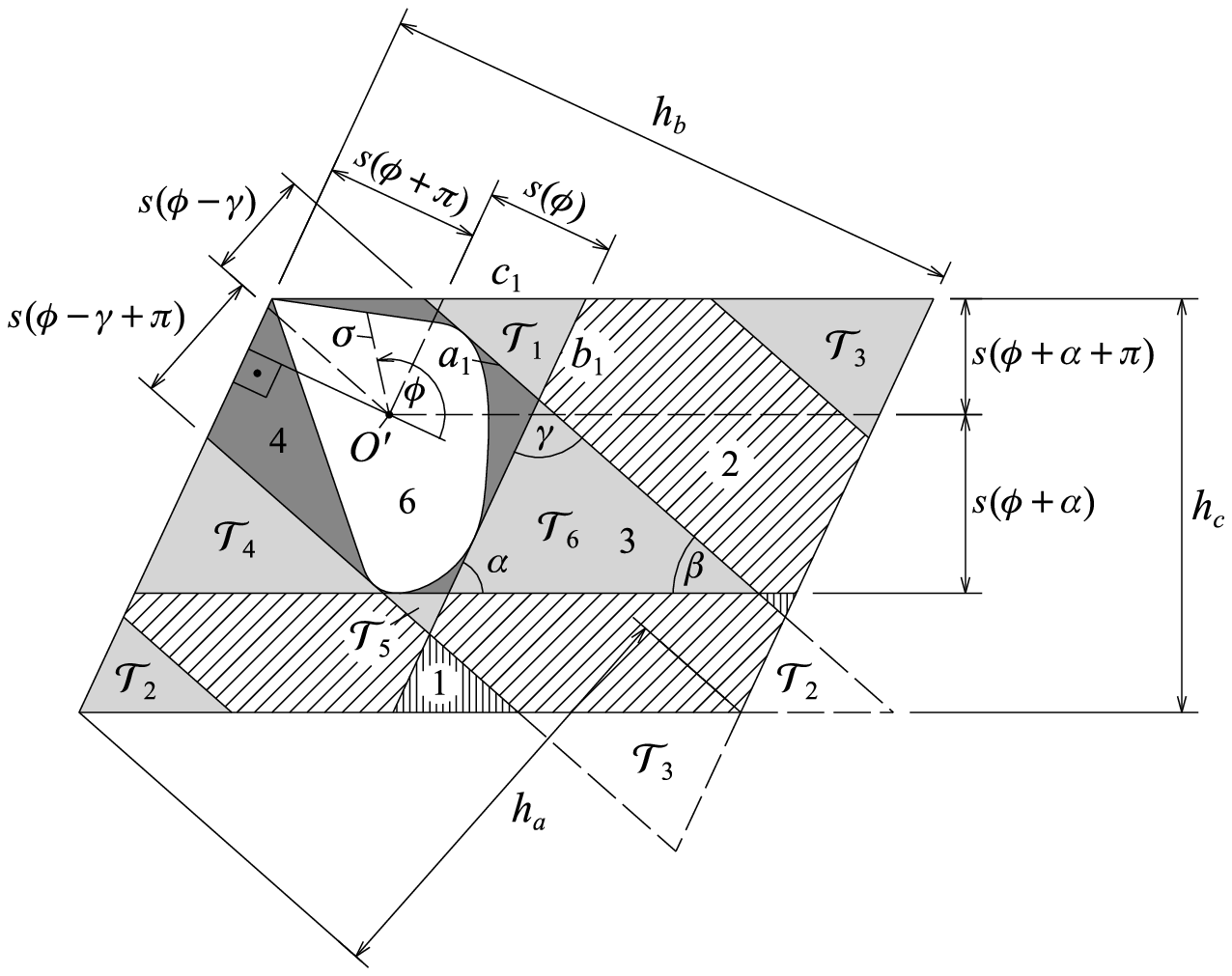}
  \end{center}
  \vspace{-0.4cm}
  \caption{\label{Qnew} Rearrangement of $\Q$}
\end{figure}

\clearpage

\noindent
For every $\phi$, $F_6(\phi)$ is equal to $F$ since set~6 is a congruent copy of $\C$ (see Fig.\:\ref{Q} and \ref{Qnew}), hence
\beq
  p(6) = \frac{1}{2\pi Q}\int_0^{2\pi}F\,\dd\phi
       = \frac{F}{2\pi Q}\int_0^{2\pi}\dd\phi
       = \frac{F}{Q}
\eeq
(cf. \cite[p.\:321]{Ren_Zhang}, \cite[pp.\:73-74]{Ren}). Now, we determine the probability $p(3)$. $\C$, with direction $\phi$, hits exactly three triangles of $\R$ if $C\in\T_1(\phi)\cup\T_2(\phi)\cup\ldots\cup\T_6(\phi)$. Every triangle $\T_i(\phi)$ is similar to the triangle $\T$. We denote by $a_i(\phi)$, $b_i(\phi)$, $c_i(\phi)$ the sides of $\T_i(\phi)$ with opposite angles $\alpha$, $\beta$, $\gamma$, respectively. Because of the said similarity, the area of $\T_i(\phi)$ is given by    
\beq
  \frac{a_i^2(\phi)}{a^2}\,\frac{Q}{2} 
	= \frac{b_i^2(\phi)}{b^2}\,\frac{Q}{2}
	= \frac{c_i^2(\phi)}{c^2}\,\frac{Q}{2}\,,
\eeq
and, therefore,
\beq
  F_3(\phi) = \frac{Q}{2}\left(
	\frac{a_1^2(\phi)}{a^2}+\frac{c_2^2(\phi)}{c^2}+\frac{b_3^2(\phi)}{b^2} +
	\frac{a_4^2(\phi)}{a^2}+\frac{c_5^2(\phi)}{c^2}+\frac{b_6^2(\phi)}{b^2}\right),
\eeq
hence
\begin{align*}
  p(3)
  = {} & \frac{1}{4\pi}\bigg(
		\frac{1}{a^2}\int_0^{2\pi}a_1^2(\phi)\,\dd\phi +	
		\frac{1}{c^2}\int_0^{2\pi}c_2^2(\phi)\,\dd\phi +	
		\frac{1}{b^2}\int_0^{2\pi}b_3^2(\phi)\,\dd\phi\\
       &	\;\;\;+\frac{1}{a^2}\int_0^{2\pi}a_4^2(\phi)\,\dd\phi +	
		\frac{1}{c^2}\int_0^{2\pi}c_5^2(\phi)\,\dd\phi +	
		\frac{1}{b^2}\int_0^{2\pi}b_6^2(\phi)\,\dd\phi\bigg)\,.	
\end{align*}
Using the support function $s$ of $\C$, one finds
\beq
  a_1(\phi) 
	= s(\phi-\beta-\gamma)\csc\beta+s(\phi)\csc\gamma-s(\phi-\gamma)(\cot\beta+\cot\gamma)\,.
\eeq
We put $d(\beta,\gamma,\phi):=a_1(\phi)$. Due to the symmetry of $\R$,
\beq
 \begin{array}{l@{\;=\;}l@{\quad}l@{\;=\;}l}
  c_2(\phi) & d(\alpha,\beta,\phi-\gamma)\,, &
  b_3(\phi) & d(\gamma,\alpha,\phi-\gamma-\beta)\,,\\[0.11cm]
  a_4(\phi) & d(\beta,\gamma,\phi+\alpha+\beta+\gamma)\,, &
  c_5(\phi) & d(\alpha,\beta,\phi+\alpha+\beta)\,,\\[0.13cm]
  b_6(\phi) & d(\gamma,\alpha,\phi+\alpha)\,. 
 \end{array}
\eeq
Since $d(x,y,\,\cdot\,)$ with $(x,y)\in\{(\alpha,\beta),(\beta,\gamma),(\gamma,\alpha)\}$ is a $2\pi$-periodic function, we get
\beq
 \begin{array}{c@{\;=\;}c@{\;=\;}c}
  \ds{\int_0^{2\pi}a_1^2(\phi)\,\dd\phi} &
  \ds{\int_0^{2\pi}d^2(\beta,\gamma,\phi)\,\dd\phi} & 
  \ds{\int_0^{2\pi}a_4^2(\phi)\,\dd\phi\,,}\\[0.35cm]
  \ds{\int_0^{2\pi}b_3^2(\phi)\,\dd\phi} &
  \ds{\int_0^{2\pi}d^2(\gamma,\alpha,\phi)\,\dd\phi} & 
  \ds{\int_0^{2\pi}b_6^2(\phi)\,\dd\phi\,,}\\[0.35cm]  
  \ds{\int_0^{2\pi}c_2^2(\phi)\,\dd\phi} &
  \ds{\int_0^{2\pi}d^2(\alpha,\beta,\phi)\,\dd\phi} & 
  \ds{\int_0^{2\pi}c_5^2(\phi)\,\dd\phi\,,}  
 \end{array}
\eeq
hence
\begin{align*}
  p(3) = {} & \frac{1}{2\pi}\bigg(\frac{H(\beta,\gamma)}{a^2} + 
		\frac{H(\gamma,\alpha)}{b^2} + \frac{H(\alpha,\beta)}{c^2}\bigg)\,,
\end{align*}
where
\beq
  H(x,y):=\int_0^{2\pi}d^2(x,y,\phi)\,\dd\phi\,.
\eeq
With
\begin{align*}
  & \!\!\!\!\!\!\!
  d^2(\beta,\gamma,\phi+\beta+\gamma)\\
  = {} & s^2(\phi)\csc^2\beta+s^2(\phi+\beta+\gamma)\csc^2\gamma
		+s^2(\phi+\beta)(\cot\beta+\cot\gamma)^2\\
       & +2s(\phi)s(\phi+\beta+\gamma)\csc\beta\csc\gamma
		-2s(\phi)s(\phi+\beta)\csc\beta\,(\cot\beta+\cot\gamma)\\
       &	-2s(\phi+\beta+\gamma)s(\phi+\beta)\csc\gamma\,(\cot\beta+\cot\gamma)\,,
\end{align*}
and, using the the relations 
\begin{align*}
  & \!\!\!\!\!\!\!\!\!\!\!\!\!\!\!\!\!\!\!\!\!\!\!\!\!\!\!\!\!\!\!\!\!\!\!\!\!
    \!\!\!\!\!\!\!\!\!\!\!\!\!\!\!\!\!\!\!\!\!\!\!\!\!\!\!\!\!\!\!\!\!\!\!\!\!	
    \int_0^{2\pi}s^2(\phi+\beta)\:\dd\phi 
    = \int_0^{2\pi}s^2(\phi+\beta+\gamma)\:\dd\phi
    = \int_0^{2\pi}s^2(\phi)\:\dd\phi\,,\\[0.2cm]
  & \!\!\!\!\!\!\!\!\!\!\!\!\!\!\!\!\!\!\!\!\!\!\!\!\!\!\!\!\!\!\!\!\!\!\!\!\!
    \!\!\!\!\!\!\!\!\!\!\!\!\!\!\!\!\!\!\!\!\!\!\!\!\!\!\!\!\!\!\!\!\!\!\!\!\!
    \int_0^{2\pi}s(\phi+\beta)s(\phi+\beta+\gamma)\,\dd\phi
    = \int_0^{2\pi}s(\phi)s(\phi+\gamma)\,\dd\phi\,,\\[0.2cm]
    \int_0^{2\pi}s(\phi)\,s(\phi+\beta+\gamma)\:\dd\phi
  = {} & \int_0^{2\pi}s(\phi+\alpha)\,s(\phi+\alpha+\beta+\gamma)\:\dd\phi\\
  = {} & \int_0^{2\pi}s(\phi+\pi)\,s(\phi+\alpha)\:\dd\phi\,,
\end{align*}  
that result from the $2\pi$-periodicity of $s$, one finds
\begin{align*}
  H(\beta,\gamma) 
  = {} & \int_0^{2\pi}d^2(\beta,\gamma,\phi)\:\dd\phi
  = \int_0^{2\pi}d^2(\beta,\gamma,\phi+\beta+\gamma)\:\dd\phi\displaybreak[0]\\ 
  = {} & \left[\csc^2\beta+\csc^2\gamma+(\cot\beta+\cot\gamma)^2\right]
		\int_0^{2\pi}s^2(\phi)\:\dd\phi\\
       & +2\csc\beta\csc\gamma\int_0^{2\pi}s(\phi+\pi)\,s(\phi+\alpha)\:\dd\phi\\
       & -2(\cot\beta+\cot\gamma)\csc\beta\int_0^{2\pi}s(\phi)\,s(\phi+\beta)\:\dd\phi\\
       & -2(\cot\beta+\cot\gamma)\csc\gamma\int_0^{2\pi}s(\phi)\,s(\phi+\gamma)\:\dd\phi\,.
\end{align*}
From
\begin{align*}
  & \!\!\!\!\!\!\! \int_0^{2\pi}w(\phi)\,w(\phi+\alpha)\:\dd\phi\\[0.1cm]
  = {} & \int_0^{2\pi}[s(\phi)+s(\phi+\pi)][s(\phi+\alpha)
		+s(\phi+\alpha+\pi)]\:\dd\phi\displaybreak[0]\\[0.1cm]
  = {} & \int_0^{2\pi}[s(\phi)\,s(\phi+\alpha)+s(\phi)\,s(\phi+\alpha+\pi)
		+s(\phi+\pi)\,s(\phi+\alpha)\\[0.1cm]
       & +s(\phi+\pi)\,s(\phi+\alpha+\pi)]\:\dd\phi\\[0.1cm]
  = {} & 2\int_0^{2\pi}s(\phi)\,s(\phi+\alpha)\:\dd\phi 
		+2\int_0^{2\pi}s(\phi+\pi)\,s(\phi+\alpha)\:\dd\phi\,,
\end{align*}
we get, with the $\pi$-periodicity of $w$,
\begin{align*}
  & \!\!\!\!\!\!\! \int_0^{2\pi}s(\phi+\pi)\,s(\phi+\alpha)\:\dd\phi\\[0.1cm]
  = {} & \frac{1}{2}\int_0^{2\pi}w(\phi)\,w(\phi+\alpha)\:\dd\phi
 		-\int_0^{2\pi}s(\phi)\,s(\phi+\alpha)\:\dd\phi\\[0.1cm]
  = {} & \int_0^{\pi}w(\phi)\,w(\phi+\alpha)\:\dd\phi
 		-\int_0^{2\pi}s(\phi)\,s(\phi+\alpha)\:\dd\phi\,.
\end{align*}
Therefore, with the abbreviations
\beq
  I(x):=\int_0^{\pi}w(\phi)\,w(\phi+\alpha)\:\dd\phi \quad\mbox{and}\quad
  J(x):=\int_0^{2\pi}s(\phi)\,s(\phi+\alpha)\:\dd\phi\,,
\eeq
we find
\begin{align*}
  \frac{H(\beta,\gamma)}{a^2}
  = {} & \frac{1}{a^2}\int_0^{2\pi}d^2(\beta,\gamma,\phi)\:\dd\phi
  = \frac{1}{a^2}\int_0^{2\pi}d^2(\beta,\gamma,\phi+\beta+\gamma)\:\dd\phi\\[0.2cm] 
  = {} & \left[\csc^2\beta+\csc^2\gamma+(\cot\beta+\cot\gamma)^2\right]J(0)\\[0.1cm]
       & +2\csc\beta\csc\gamma\,[I(\alpha)-J(\alpha)]
		-2(\cot\beta+\cot\gamma)\csc\beta\,J(\beta)\\[0.1cm]
       &	-2(\cot\beta+\cot\gamma)\csc\gamma\,J(\gamma)\displaybreak[0]\\[0.1cm]
  = {} & \left[\dfrac{c^2}{(ac\sin\beta)^2}
			+\dfrac{b^2}{(ab\sin\gamma)^2}+\left(\dfrac{c\cos\beta}{ac\sin\beta}
			+\dfrac{b\cos\gamma}{ab\sin\gamma}\right)^2\right]J(0)\\[0cm]
       & +\dfrac{2bc}{ac\sin\beta\,ab\sin\gamma}\,I(\alpha)
			-\dfrac{2bc}{ac\sin\beta\,ab\sin\gamma}\,J(\alpha)\\[0.1cm]
       & -2\left(\dfrac{c\cos\beta}{ac\sin\beta}+\dfrac{b\cos\gamma}{ab\sin\gamma}\right)
			\dfrac{c}{ac\sin\beta}\,J(\beta)\\[0.1cm]
       & -2\left(\dfrac{c\cos\beta}{ac\sin\beta}+\dfrac{b\cos\gamma}{ab\sin\gamma}\right)
			\dfrac{b}{ab\sin\gamma}\,J(\gamma)\\[0.2cm]
  = {} & \dfrac{(a^2+b^2+c^2)J(0)}{Q^2}+\dfrac{2bcI(\alpha)}{Q^2}-\dfrac{2bcJ(\alpha)}{Q^2}
			-\dfrac{2acJ(\beta)}{Q^2}-\dfrac{2abJ(\gamma)}{Q^2}\,. 	
\end{align*}
It follows that
\begin{align*}
  p(3)
  = {} & \frac{1}{2\pi Q^2}\big[(a^2+b^2+c^2)J(0)+2bcI(\alpha)-2bcJ(\alpha)-2caJ(\beta)
			-2abJ(\gamma)\\[-0.12cm]
       & \quad\;\;\, +(b^2+c^2+a^2)J(0)+2caI(\beta)-2caJ(\beta)\:\!\!
			-2abJ(\gamma)-2bcJ(\alpha)\\[0.12cm]
       & \quad\;\;\, +(c^2+a^2+b^2)J(0)+2abI(\gamma)-2abJ(\gamma)
			-2bcJ(\alpha)-2caJ(\beta)]\\[0.15cm]
  = {} & \frac{3(a^2+b^2+c^2)J(0)}{2\pi Q^2}+\frac{bcI(\alpha)+caI(\beta)+abI(\gamma)}{\pi Q^2}\\
       & -\frac{3\left[bcJ(\alpha)+caJ(\beta)+abJ(\gamma)\right]}{\pi Q^2}\,. 
\end{align*}
As the next step, we determine $p(4)$. Using Fig.\:\ref{Qnew}, one sees that
\begin{align*}
  F_4(\phi) 
  = {} & \frac{w(\phi)w(\phi+\alpha)}{\sin\alpha}-\frac{a_1^2(\phi)}{a^2}\,\frac{Q}{2}
		-\frac{a_4^2(\phi)}{a^2}\,\frac{Q}{2}-F\\[0.1cm]
  = {} & \frac{bcw(\phi)w(\phi+\alpha)}{Q}-\frac{a_1^2(\phi)}{a^2}\,\frac{Q}{2}
		-\frac{a_4^2(\phi)}{a^2}\,\frac{Q}{2}-F\,.
\end{align*}
So we find
\begin{align*}
  p(4)
  = {} & \frac{1}{2\pi Q}\bigg[\frac{bc}{Q}\int_0^{2\pi}w(\phi)w(\phi+\alpha)\,\dd\phi
		-\frac{Q}{2a^2}\bigg(\int_0^{2\pi}a_1^2(\phi)\,\dd\phi\\
       &	+\int_0^{2\pi}a_4^2(\phi)\,\dd\phi\bigg)-F\int_0^{2\pi}\dd\phi\bigg]\displaybreak[0]\\
  = {} & \frac{1}{2\pi Q}\bigg[\frac{2bcI(\alpha)}{Q}
		-\frac{Q}{a^2}\int_0^{2\pi}d^2(\beta,\gamma,\phi)\,\dd\phi-2\pi F\bigg]\\
  = {} & \frac{bcI(\alpha)}{\pi Q^2}-\frac{H(\beta,\gamma)}{2\pi a^2}-\frac{F}{Q}\,.
\end{align*}
Taking into account the symmetry of $\R$, we also have
\begin{align*}
  p(4)
  = {} & \frac{1}{3}\bigg[
		\bigg(\frac{bcI(\alpha)}{\pi Q^2}-\frac{H(\beta,\gamma)}{2\pi a^2}-\frac{F}{Q}\bigg)
        +\bigg(\frac{caI(\beta)}{\pi Q^2}-\frac{H(\gamma,\alpha)}{2\pi b^2}-\frac{F}{Q}
			\bigg)\\[0.1cm]
       &+\bigg(\frac{abI(\gamma)}{\pi Q^2}-\frac{H(\alpha,\beta)}{2\pi c^2}-\frac{F}{Q}
			\bigg)\bigg]\\[0.15cm]
  = {} & \frac{bcI(\alpha)+caI(\beta)+abI(\gamma)}{3\pi Q^2}-\frac{1}{3}\,p(3)-\frac{F}{Q}\\[0.1cm]
  = {} & \frac{bcI(\alpha)+caI(\beta)+abI(\gamma)}{3\pi Q^2}-\frac{1}{3}\bigg(
		\frac{3(a^2+b^2+c^2)J(0)}{2\pi Q^2}\\[0.1cm]
       & +\frac{bcI(\alpha)+caI(\beta)+abI(\gamma)}{\pi Q^2}
		-\frac{3[bcJ(\alpha)+caJ(\beta)+abJ(\gamma)]}{\pi Q^2}\bigg)-\frac{F}{Q}\\
  = {} & \frac{bcJ(\alpha)+caJ(\beta)+abJ(\gamma)}{\pi Q^2}-\frac{(a^2+b^2+c^2)J(0)}{2\pi Q^2}
		-\frac{F}{Q}\,.
\end{align*}
The next probability to determine is $p(1)$. One finds
\begin{align*}
  F_1(\phi)
  = {} & \frac{[h_b-w(\phi)][h_c-w(\phi+\alpha)]}{\sin\alpha}-\bigg(\frac{[h_c-w(\phi+\alpha)]
		w(\phi-\gamma)}{\sin\beta}\\
       & -\frac{c_2^2(\phi)}{c^2}\frac{Q}{2}-\frac{c_5^2(\phi)}{c^2}\frac{Q}{2}\bigg)\\
  = {} & \frac{h_bh_c}{\sin\alpha}-\frac{h_bw(\phi+\alpha)}{\sin\alpha}-\frac{h_cw(\phi)}{\sin\alpha}
		+\frac{w(\phi)w(\phi+\alpha)}{\sin\alpha}-\frac{h_cw(\phi-\gamma)}{\sin\beta}\\
       & +\frac{w(\phi+\alpha)w(\phi-\gamma)}{\sin\beta}+\frac{Q}{2c^2}\left[
		d^2(\alpha,\beta,\phi-\gamma)+d^2(\alpha,\beta,\phi+\alpha+\beta)\right]\,.
\end{align*}
Using $h_b=a\sin\gamma=c\sin\alpha$, $h_c=b\sin\alpha=a\sin\beta$, $w(\phi-\gamma)=w(\phi+\alpha+\beta)$, and $Q=ab\sin\gamma$ we get
\begin{align*}
  F_1(\phi)
  = {} & Q-cw(\phi+\alpha)-bw(\phi)+\frac{bc}{Q}\,w(\phi)w(\phi+\alpha)-aw(\phi+\alpha+\beta)
			\displaybreak[0]\\[-0.05cm]
       &	+\frac{ac}{Q}\,w(\phi+\alpha)w(\phi+\alpha+\beta)+\frac{Q}{2c^2}
			\left[d^2(\alpha,\beta,\phi-\gamma)\right.\displaybreak[0]\\[0.12cm]
       & \left.+\:d^2(\alpha,\beta,\phi+\alpha+\beta)\right]. 
\end{align*}
This, with
\begin{align*}
  \int_0^\pi w(\phi)\,\dd\phi
  = {} & \int_0^\pi [s(\phi)+s(\phi+\pi)]\,\dd\phi
  = \int_0^{2\pi}s(\phi)\,\dd\phi = u
\end{align*}
(see \cite[p.\:3]{Santalo2}), and
\begin{align*}
  \int_0^{2\pi}w(\phi+\alpha)w(\phi+\alpha+\beta)\,\dd\phi
  = {} & 2\int_0^{\pi}w(\phi+\alpha)w(\phi+\alpha+\beta)\,\dd\phi\\
  = {} & 2\int_0^{\pi}w(\phi)w(\phi+\beta)\,\dd\phi\,,
\end{align*}
gives
\begin{align*}
  p(1)
  = {} & \frac{1}{2\pi Q}\bigg(Q\int_0^{2\pi}\dd\phi-2(a+b+c)\int_0^{\pi}w(\phi)\,\dd\phi\\
       &	+\frac{2bc}{Q}\int_0^{\pi}w(\phi)w(\phi+\alpha)\,\dd\phi
		+\frac{2ac}{Q}\int_0^{\pi}w(\phi)w(\phi+\beta)\,\dd\phi\\
       & +\frac{Q}{c^2}\int_0^{2\pi}d^2(\alpha,\beta,\phi)\,\dd{\phi}\bigg)\\
  = {} & 1-\frac{(a+b+c)u}{\pi Q}+\frac{bcI(\alpha)}{\pi Q^2}+\frac{caI(\beta)}{\pi Q^2}
			+\frac{H(\alpha,\beta)}{2\pi c^2}\,.
\end{align*}
We now obtain, with the symmetry of $\R$,
\begin{align*}
  p(1)
  = {} & \frac{1}{3}\bigg(1-\frac{(a+b+c)u}{\pi Q}+\frac{bcI(\alpha)}{\pi Q^2}
		+\frac{caI(\beta)}{\pi Q^2}+\frac{H(\alpha,\beta)}{2\pi c^2}
			\displaybreak[0]\\
       & \:+1-\frac{(b+c+a)u}{\pi Q}+\frac{caI(\beta)}{\pi Q^2}+\frac{abI(\gamma)}{\pi Q^2}
		+\frac{H(\beta,\gamma)}{2\pi a^2}\displaybreak[0]\\
       & \;+1-\frac{(c+a+b)u}{\pi Q}+\frac{abI(\gamma)}{\pi Q^2}+\frac{bcI(\alpha)}{\pi Q^2}
		+\frac{H(\gamma,\alpha)}{2\pi b^2}\bigg)\\
  = {} & 1-\frac{(a+b+c)u}{\pi Q}+\frac{2[bcI(\alpha)+caI(\beta)+abI(\gamma)]}{3\pi Q^2}
		+\frac{1}{3}\,p(3)\\
  = {} & 1-\frac{(a+b+c)u}{\pi Q}+\frac{(a^2+b^2+c^2)J(0)}{2\pi Q^2}
		+\frac{bcI(\alpha)+caI(\beta)+abI(\gamma)}{\pi Q^2}\\
       & -\frac{bcJ(\alpha)+caJ(\beta)+abJ(\gamma)}{\pi Q^2}\,.
\end{align*}
Now we determine the expression for $p(2)$. Let $E_2$ denote the event that $\C$ hits exactly two triangles of $\R$; let further $E_a$, $E_b$, and $E_c$ denote the events that $\C$ hits side $a$, $b$, and $c$, respectively. Since the events $E_2\cap E_a$, $E_2\cap E_b$ and $E_2\cap E_c$ are pairwise disjoint,
\beq
  p(2) = P(E_2) = P(E_2\cap E_a)+P(E_2\cap E_b)+P(E_2\cap E_c)\,.
\eeq
We have 
\beq
  P(E_2\cap E_a) = \int_0^{2\pi}P(E_2\cap E_a\,|\,\phi)\,\frac{\dd\phi}{2\pi} \,,\quad
  P(E_2\cap E_a\,|\,\phi) = \frac{F_{2,\,a}(\phi)}{Q}\,.
\eeq
For the area $F_{2,\,a}(\phi)$, we can write
\begin{align*}
  F_{2,\,a}(\phi)
  = {} & \frac{w(\phi-\gamma)[h_b-w(\phi)]}{\sin\gamma}-\frac{b_3^2(\phi)}{b^2}\frac{Q}{2}
		-\frac{b_6^2(\phi)}{b^2}\frac{Q}{2}\\
  = {} & \frac{h_bw(\phi-\gamma)}{\sin\gamma}-\frac{w(\phi-\gamma)w(\phi)}{\sin\gamma}
		-\frac{Q}{2b^2}\left[d^2(\gamma,\alpha,\phi-\gamma-\beta)\right.\\
       & \left.+\:d^2(\gamma,\alpha,\phi+\alpha)\right]\displaybreak[0]\\
  = {} & aw(\phi-\gamma)-\frac{ab}{Q}\,w(\phi-\gamma)w(\phi)
		-\frac{Q}{2b^2}\left[d^2(\gamma,\alpha,\phi-\gamma-\beta)\right.\\
       & \left.+\:d^2(\gamma,\alpha,\phi+\alpha)\right].
\end{align*}
Hence, with $\int_0^{2\pi}w(\phi-\gamma)w(\phi)\,\dd\phi=2\int_0^{\pi}w(\phi)w(\phi+\gamma)\,\dd\phi$,
\begin{align*}
  P(E_2\cap E_a) 
  = {} & \frac{1}{2\pi Q}\left[2a\int_0^{\pi}w(\phi)\,\dd\phi-\frac{2ab}{Q}
		\int_0^{\pi}w(\phi)w(\phi+\gamma)\,\dd\phi\right.\\
       &	\left.-\:\frac{Q}{b^2}\int_0^{2\pi}d^2(\gamma,\alpha,\phi)\,\dd\phi\right]
  = \frac{au}{\pi Q}-\frac{abI(\gamma)}{\pi Q^2}-\frac{H(\gamma,\alpha)}{2\pi b^2}\,.
\end{align*}
Due to the symmetry of $\R$, we also have
\begin{align*}
  P(E_2\cap E_b) 
  = {} & \frac{bu}{\pi Q}-\frac{bcI(\alpha)}{\pi Q^2}-\frac{H(\alpha,\beta)}
		{2\pi c^2}\,,\\[0.1cm]
  P(E_2\cap E_c) 
  = {} & \frac{cu}{\pi Q}-\frac{caI(\beta)}{\pi Q^2}-\frac{H(\beta,\gamma)}{2\pi a^2}	\,.
\end{align*}
It follows that
\begin{align*}
  p(2)
  = {} & \frac{(a+b+c)u}{\pi Q}-\frac{bcI(\alpha)+caI(\beta)+abI(\gamma)}{\pi Q^2}-p(3)\\
  = {} & \frac{(a+b+c)u}{\pi Q}-\frac{3(a^2+b^2+c^2)J(0)}{2\pi Q^2}
		-\frac{2[bcI(\alpha)+caI(\beta)+abI(\gamma)]}{\pi Q^2}\\
       & +\frac{3[bcJ(\alpha)+caJ(\beta)+abJ(\gamma)]}{\pi Q^2}\,. 
\end{align*}
In order to calculate the expectation $\E=\sum_{i=1}^6 i\,p(i)$ for the random number $Z$ of hit triangles, we put
\begin{align*}
  L := {} & \frac{(a^2+b^2+c^2)J(0)}{2\pi Q^2}\,,\quad
  M := \frac{bcI(\alpha)+caI(\beta)+abI(\gamma)}{\pi Q^2}\,,\\
  N := {} & \frac{bcJ(\alpha)+caJ(\beta)+abJ(\gamma)}{\pi Q^2}\,.
\end{align*}
So we have
\begin{align*}
  p(1) = {} & 1-\frac{(a+b+c)u}{\pi Q}+L+M-N\,,\\
  p(2) = {} & \frac{(a+b+c)u}{\pi Q}-3L-2M+3N\,,\\
  p(3) = {} & 3L+M-3N\,,\quad
  p(4) = N-L-\frac{F}{Q}\,,\quad
  p(5) = 0\,,\quad p(6)=\frac{F}{Q}\,,
\end{align*}
and find
\begin{align*}
  \E 
  = {} & 1-\frac{(a+b+c)u}{\pi Q}+L+M-N\\
       &	+2\left(\frac{(a+b+c)u}{\pi Q}-3L-2M+3N\right)
		+3(3L+M-3N)\\
       & +4\left(N-L-\frac{F}{Q}\right)+5\cdot 0+6\,\frac{F}{Q}
  = 1+\frac{(a+b+c)u}{\pi Q}+\frac{2F}{Q}\,.
\qedhere	
\end{align*}
\end{proof}

\section{Convex bodies with $w(\phi)=2s(\phi)$}

For all convex bodies $\C$ where the reference point $O$ may be choosen such that $w(\phi)=2s(\phi)$ for every $\phi\in[0,2\pi)$, we have
\begin{align*}
  I(x) = {} & \int_0^\pi w(\phi)w(\phi+x)\,\dd\phi 
       = \frac{1}{2}\int_0^{2\pi} w(\phi)w(\phi+x)\,\dd\phi\displaybreak[0]\\
       =	{} & \frac{1}{2}\int_0^{2\pi} 2s(\phi)\,2s(\phi+x)\,\dd\phi
		     = 2\int_0^{2\pi}s(\phi)\,s(\phi+x)\,\dd\phi	= 2J(x)\,, 
\end{align*}
and, therefore, get the formulas of the following corollary which are even simpler to compute.

\begin{corollary}  \label{w=2s}
If condition $\eqref{max}$ holds true and $w(\phi)=2s(\phi)$ for every $\phi\in[0,2\pi)$, then
\begin{align*}
  p(1) = {} & 1 - \frac{(a+b+c)u}{\pi Q} + \frac{(a^2+b^2+c^2)I(0)}{4\pi Q^2}
		     + \frac{bcI(\alpha)+caI(\beta)+abI(\gamma)}{2\pi Q^2}\,,\\[0.15cm]
  p(2) = {} & \frac{(a+b+c)u}{\pi Q} - \frac{3(a^2+b^2+c^2)I(0)}{4\pi Q^2}
		     - \frac{bcI(\alpha)+caI(\beta)+abI(\gamma)}{2\pi Q^2}\,,
			\displaybreak[0]\\[0.15cm]
  p(3) = {} & \frac{3(a^2+b^2+c^2)I(0)}{4\pi Q^2}
		     - \frac{bcI(\alpha)+caI(\beta)+abI(\gamma)}{2\pi Q^2}\,,
			\displaybreak[0]\\[0.15cm]
  p(4) = {} & \frac{bcI(\alpha)+caI(\beta)+abI(\gamma)}{2\pi Q^2}
		     - \frac{(a^2+b^2+c^2)I(0)}{4\pi Q^2} - \frac{F}{Q}\,,
			\displaybreak[0]\\[0.15cm]
  p(5) = {} & 0 \,,\quad p(6) = \frac{F}{Q}\,.                      
\end{align*}
\end{corollary}

\noindent In the following, we give some examples.

\subsection{Rectangles}

\begin{corollary} \label{R1}
Let $\R$ be a lattice of acute or right triangles. Let $\C$ be a rectangle with side lengths $g$ and $h$, and statisfying condition \eqref{max}. Then
\begin{align*}
  p(1) = {} & 1-\frac{(a+b+c)u}{\pi Q}+\frac{(\alpha a^2+\beta b^2+\gamma c^2)(g^2+h^2)}{2\pi Q^2}
		     +\frac{3(g^2+h^2)}{2\pi Q}\displaybreak[0]\\
		{} & +\frac{(a^2+b^2+c^2)F}{\pi Q^2}+\frac{F}{Q}\,,
				\displaybreak[0]\\[0.15cm]
  p(2) = {} & \frac{(a+b+c)u}{\pi Q} - \frac{(a^2+b^2+c^2)(g^2+h^2)}{4Q^2}
		     -\frac{(\alpha a^2+\beta b^2+\gamma c^2)(g^2+h^2)}{2\pi Q^2}\\
		{} & -\frac{3(g^2+h^2)}{2\pi Q}-\frac{2(a^2+b^2+c^2)F}{\pi Q^2}-\frac{F}{Q}\,,
				\displaybreak[0]\\[0.15cm]
  p(3) = {} & \frac{(a^2+b^2+c^2)(g^2+h^2)}{2 Q^2} 
		     -\frac{(\alpha a^2+\beta b^2+\gamma c^2)(g^2+h^2)}{2\pi Q^2}
		     -\frac{3(g^2+h^2)}{2\pi Q}\\
		{} & +\frac{(a^2+b^2+c^2)F}{\pi Q^2}-\frac{F}{Q}\,,\displaybreak[0]\\[0.15cm]
  p(4) = {} & \frac{3(g^2+h^2)}{2\pi Q}-\frac{(a^2+b^2+c^2)(g^2+h^2)}{4Q^2} 
		     +\frac{(\alpha a^2+\beta b^2+\gamma c^2)(g^2+h^2)}{2\pi Q^2}\,,
				\displaybreak[0]\\[0.15cm]
  p(5) = {} & 0 \,,\quad p(6) = \frac{F}{Q}\,,                      
\end{align*}
with $u=2(g+h)$ and $F=gh$.
\end{corollary}

\begin{proof}
The width of the rectangle in the direction $\phi$ is given by
\beq
  w(\phi) = g\left|\cos\phi\right|+h\left|\sin\phi\right|.
\eeq
$w$ has the restriction
\beq
  w(\phi)|_{0\leq\phi<2\pi} = \left\{\ba{@{\:}rcr@{\;}c@{\;}l}
	w_1(\phi) & \mbox{if} & 0\leq & \phi & <\pi/2\,,\\[0.1cm]
	w_2(\phi) & \mbox{if} & \pi/2\leq & \phi & <\pi\,,\\[0.1cm]
	w_3(\phi) & \mbox{if} & \pi\leq & \phi & <3\pi/2\,,\\[0.1cm]
	w_4(\phi) & \mbox{if} & 3\pi/2< & \phi & <2\pi\,,  
  \ea\right.
\eeq
with
\beq
  \ba{c@{\;=\;}r@{\;}c@{\;}c}
	w_1(\phi) &  g\cos\phi & + & h\sin\phi\,,\\[0.1cm]
	w_2(\phi) & -g\cos\phi & + & h\sin\phi\,,\\[0.1cm]
	w_3(\phi) & -g\cos\phi & - & h\sin\phi\,,\\[0.1cm]
	w_4(\phi) &  g\cos\phi & - & h\sin\phi\,.
  \ea
\eeq
For the calculation of $I(x)$, $x\in\{\alpha,\beta,\gamma\}$, we have to distinguish the cases
\beq
  0\leq\phi<\pi/2 \,,\quad \pi/2\leq\phi\leq\pi\,,
\eeq
and
\beq
  x\leq\phi+x<\pi/2 \,,\quad \pi/2\leq\phi+x<\pi \,,\quad
  \pi\leq\phi+x\leq\pi+x\,.
\eeq
Since $0<x\leq\pi/2$, this yields
\beq
  0\leq\phi<\pi/2-x \,,\;\; \pi/2-x\leq\phi<\pi/2 \,,\;\;
  \pi/2\leq\phi<\pi-x \,,\;\; \pi-x\leq\phi\leq\pi\,;
\eeq
therefore
\begin{align*}
  I(x)
  = {} & \left(\int_0^{\pi/2-x}+\int_{\pi/2-x}^{\pi/2}+
		\int_{\pi/2}^{\pi-x}+\int_{\pi-x}^\pi\right)
		w(\phi)w(\phi+x)\:\dd\phi\displaybreak[0]\\
  = {} & \int_0^{\pi/2-x}w_1(\phi)\,w_1(\phi+x)\,\dd\phi
		+\int_{\pi/2-x}^{\pi/2}w_1(\phi)\,w_2(\phi+x)\,\dd\phi\\
       & +\int_{\pi/2}^{\pi-x}w_2(\phi)\,w_2(\phi+x)\,\dd\phi
		+\int_{\pi-x}^{\pi}w_2(\phi)\,w_3(\phi+x)\,\dd\phi\\
  = {} & \frac{1}{2}\left\{\left[(\pi-2x)\left(g^2+h^2\right)+4gh\right]\cos x
		+2\left(g^2+h^2+2xgh\right)\sin x\right\}.
\end{align*}
This, with \eqref{area_Q}, $F=gh$, and
\begin{align*}
  2bc\cos\alpha = {} & b^2+c^2-a^2,\;\,
  2ac\cos\beta  = a^2+c^2-b^2,\;\, 
  2ab\cos\gamma = a^2+b^2-c^2,
\end{align*}
gives
\begin{align*}
  & \!\!\!\!\!\!\! bcI(\alpha)+caI(\beta)+abI(\gamma)\\[0.1cm]
  = {} & (bc/2)\left[(\pi-2\alpha)\left(g^2+h^2\right)+4gh\right]
		\cos\alpha+bc\left(g^2+h^2+2\alpha gh\right)\sin\alpha\\
       & +(ca/2)\left[(\pi-2\beta)\left(g^2+h^2\right)+4gh\right]
		\cos\beta+ca\left(g^2+h^2+2\beta gh\right)\sin\beta\\
       & +(ab/2)\left[(\pi-2\gamma)\left(g^2+h^2\right)+4gh\right]
		\cos\gamma+ab\left(g^2+h^2+2\gamma gh\right)\sin\gamma
			\displaybreak[0]\\[0.1cm]
  = {} & (1/4)\left[(\pi-2\alpha)\left(g^2+h^2\right)+4F\right]
         \left(b^2+c^2-a^2\right)+\left(g^2+h^2+2\alpha F\right)Q\\
       & +(1/4)\left[(\pi-2\beta)\left(g^2+h^2\right)+4F\right]
         \left(c^2+a^2-b^2\right)+\left(g^2+h^2+2\beta F\right)Q\\
       & +(1/4)\left[(\pi-2\gamma)\left(g^2+h^2\right)+4F\right]
         \left(a^2+b^2-c^2\right)+\left(g^2+h^2+2\gamma F\right)Q\\[0.1cm]
  = {} & (1/4)(\pi-2\alpha)\left(g^2+h^2\right)\left(b^2+c^2-a^2\right)\\
       & +(1/4)(\pi-2\beta)\left(g^2+h^2\right)\left(c^2+a^2-b^2\right)\\
       & +(1/4)(\pi-2\gamma)\left(g^2+h^2\right)\left(a^2+b^2-c^2\right)\\
       &	+\left(a^2+b^2+c^2\right)F+3\left(g^2+h^2\right)Q+2\pi FQ
			\displaybreak[0]\\[0.1cm]
  = {} & -(\alpha/2)\left(g^2+h^2\right)\left(a^2+b^2+c^2-2a^2\right)\\
       & -(\beta/2)\left(g^2+h^2\right)\left(a^2+b^2+c^2-2b^2\right)\\
       & -(\gamma/2)\left(g^2+h^2\right)\left(a^2+b^2+c^2-2c^2\right)\\
       &	+(\pi/4)\left(a^2+b^2+c^2\right)\left(g^2+h^2\right)
		+\left(a^2+b^2+c^2\right)F\\
       & +3\left(g^2+h^2\right)Q+2\pi FQ\displaybreak[0]\\[0.1cm]
  = {} & \left(\alpha a^2+\beta b^2+\gamma c^2\right)\left(g^2+h^2\right)
		-(\pi/4)\left(a^2+b^2+c^2\right)\left(g^2+h^2\right)\\
       &	+\left(a^2+b^2+c^2\right)F+3\left(g^2+h^2\right)Q+2\pi FQ\,.
\end{align*}
Furthermore, one finds
\beq
  I(0) = \frac{1}{2}\left[\pi\left(g^2+h^2\right)+4gh\right]
       = \frac{1}{2}\left[\pi\left(g^2+h^2\right)+4F\right].
\eeq
With Corollary \ref{w=2s}, the result follows.
\end{proof}

\begin{corollary} \label{R2}
Let $\R$ be a lattice of obtuse or right triangles with $\pi/2\leq\alpha<\pi$. Let $\C$ be a rectangle with side lengths $g$ and $h$, and statisfying condition \eqref{max}. Then
\begin{align*}
  p(1) 
  = {} & 1-\frac{(a+b+c)u}{\pi Q}
		+\frac{(\alpha a^2+\beta b^2+\gamma c^2)(g^2+h^2)}{2\pi Q^2}
		+\frac{3(g^2+h^2)}{2\pi Q}\displaybreak[0]\\
       & +\frac{2a^2F}{\pi Q^2}+\frac{2F}{Q}-\frac{2\alpha F}{\pi Q}\,,
				\displaybreak[0]\\[0.15cm]
  p(2) 
  = {} & \frac{(a+b+c)u}{\pi Q}-\frac{(a^2+b^2+c^2)(g^2+h^2)}{4Q^2}
		-\frac{(\alpha a^2+\beta b^2+\gamma c^2)(g^2+h^2)}{2\pi Q^2}\\
       & -\frac{3(g^2+h^2)}{2\pi Q}-\frac{(3a^2+b^2+c^2)F}{\pi Q^2}
		-\frac{2F}{Q}+\frac{2\alpha F}{\pi Q}\,,
				\displaybreak[0]\\[0.15cm]
  p(3)
  = {} & \frac{(a^2+b^2+c^2)(g^2+h^2)}{2 Q^2} 
		-\frac{(\alpha a^2+\beta b^2+\gamma c^2)(g^2+h^2)}{2\pi Q^2}
		-\frac{3(g^2+h^2)}{2\pi Q}\\
       & +\frac{2(b^2+c^2)F}{\pi Q^2}-\frac{2F}{Q}
		+\frac{2\alpha F}{\pi Q}\,,
				\displaybreak[0]\\[0.15cm]
  p(4)
  = {} & \frac{3(g^2+h^2)}{2\pi Q}-\frac{(a^2+b^2+c^2)(g^2+h^2)}{4Q^2} 
		+\frac{(\alpha a^2+\beta b^2+\gamma c^2)(g^2+h^2)}{2\pi Q^2}\\
       & -\frac{(b^2+c^2-a^2)F}{\pi Q^2}+\frac{F}{Q}
		-\frac{2\alpha F}{\pi Q}\,,
				\displaybreak[0]\\[0.15cm]
  p(5)
  = {} & 0 \,,\quad p(6) = \frac{F}{Q}\,.                      
\end{align*}
\end{corollary}

\begin{proof}
We have
\begin{align*}
  I(0)
  = {} & \frac{1}{2}\left[\pi\left(g^2+h^2\right)+4gh\right],\\	
  I(\beta)
  = {} & \frac{1}{2}\left\{\left[(\pi-2\beta)\left(g^2+h^2\right)+4gh\right]\cos\beta
		+2\left(g^2+h^2+2\beta gh\right)\sin\beta\right\},\\
  I(\gamma)
  = {} & \frac{1}{2}\left\{\left[(\pi-2\gamma)\left(g^2+h^2\right)+4gh\right]\cos\gamma
		\;\!+2\left(g^2+h^2+2\gamma gh\right)\sin\gamma\right\},
\end{align*}
as in the case $0<\alpha\leq\pi/2$. For the calculation of $I(\alpha)$, we have to distinguish the cases
\beq
  0\leq\phi<\pi/2 \,,\quad \pi/2\leq\phi\leq\pi\,,
\eeq
and
\beq
  \alpha\leq\phi+\alpha<\pi \,,\quad \pi\leq\phi+\alpha<3\pi/2 \,,\quad
  3\pi/2\leq\phi+\alpha\leq\pi+\alpha\,.
\eeq
This gives
\beq
  0\leq\phi<\pi-\alpha \,,\; \pi-\alpha\leq\phi<\pi/2 \,,\;
  \pi/2\leq\phi<3\pi/2-\alpha \,,\; 3\pi/2-\alpha\leq\phi\leq\pi\,;
\eeq
therefore
\begin{align*}
  I(\alpha)
  = {} & \left(\int_0^{\pi-\alpha}+\int_{\pi-\alpha}^{\pi/2}+
		\int_{\pi/2}^{3\pi/2-\alpha}+\int_{3\pi/2-\alpha}^\pi\right)
		w(\phi)w(\phi+\alpha)\:\dd\phi\displaybreak[0]\\
  = {} & \int_0^{\pi-\alpha}w_1(\phi)\,w_2(\phi+\alpha)\,\dd\phi
		+\int_{\pi-\alpha}^{\pi/2}w_1(\phi)\,w_3(\phi+\alpha)\,\dd\phi\\
       & +\int_{\pi/2}^{3\pi/2-\alpha}w_2(\phi)\,w_3(\phi+\alpha)\,\dd\phi
		+\int_{3\pi/2-\alpha}^{\pi}w_2(\phi)\,w_4(\phi+\alpha)\,\dd\phi\\
  = {} & \frac{1}{2}\,\big\{\big[(\pi-2\alpha)\left(g^2+h^2\right)
		-4gh\big]\cos\alpha+2\big[g^2+h^2\\
       & \qquad +2(\pi-\alpha)gh\big]\sin\alpha\big\}.
\end{align*}
Now, we denote the formula of $I(\alpha)$ for $0<\alpha\leq\pi/2$ with $\tilde{I}(\alpha)$, and find
\beq
  I(\alpha)-\tilde{I}(\alpha) = 2(\pi-2\alpha)gh\sin\alpha-4gh\cos\alpha\,.
\eeq
(Of course, for $\alpha=\pi/2$ we have $I(\alpha)-\tilde{I}(\alpha)=0$.) 
With
\beq
  bc\sin\alpha = Q \quad\mbox{and}\quad 
  bc\cos\alpha = \frac{1}{2}\left(b^2+c^2-a^2\right),
\eeq
we get
\begin{align*}
  \frac{bc\big[I(\alpha)-\tilde{I}(\alpha)\big]}{2\pi Q^2}
  = {} & \frac{2(\pi-2\alpha)ghQ}{2\pi Q^2}
		-\frac{2gh\left(b^2+c^2-a^2\right)}{2\pi Q^2}\\
  = {} & \frac{(\pi-2\alpha)F}{\pi Q}-\frac{(b^2+c^2-a^2)F}{\pi Q^2}\\
  = {} & \frac{F}{Q}-\frac{2\alpha F}{\pi Q}-\frac{(b^2+c^2-a^2)F}{\pi Q^2}\,,
\end{align*}
hence
\beq
  \frac{bcI(\alpha)}{2\pi Q^2} = \frac{bc\tilde{I}(\alpha)}{2\pi Q^2}
	+\frac{F}{Q}-\frac{2\alpha F}{\pi Q}-\frac{(b^2+c^2-a^2)F}{\pi Q^2}\,.
\eeq
We denote the probabilities for $0<\alpha\leq\pi/2$ by $\tilde{p}(i)$. From Corollary~\ref{w=2s}, it follows that
\begin{align*}
  p(1) = {} & \tilde{p}(1)+\frac{F}{Q}-\frac{2\alpha F}{\pi Q}
		     -\frac{(b^2+c^2-a^2)F}{\pi Q^2}\,,\displaybreak[0]\\[0.1cm]
  p(2) = {} & \tilde{p}(2)-\frac{F}{Q}+\frac{2\alpha F}{\pi Q}
		     +\frac{(b^2+c^2-a^2)F}{\pi Q^2}\,,\\[0.1cm]
  p(3) = {} & \tilde{p}(3)-\frac{F}{Q}+\frac{2\alpha F}{\pi Q}
		     +\frac{(b^2+c^2-a^2)F}{\pi Q^2}\,,\\[0.1cm]
  p(4) = {} & \tilde{p}(4)+\frac{F}{Q}-\frac{2\alpha F}{\pi Q}
		     -\frac{(b^2+c^2-a^2)F}{\pi Q^2}\,,\\[0.1cm]
  p(5) = {} & \tilde{p}(5) = 0\,,\quad p(6) = \tilde{p}(6) = \frac{F}{Q}\,.
\end{align*}
The formulae of the corollary follow easily. 
\end{proof}

\noindent
As special case of Corollaries \ref{R1} and \ref{R2}, respectively, we get the probabilities for a needle of length $\ell:=g$ with $h=0$, $u=2\ell$ and $F=0$. We find 
\begin{align*}
  p(1) 
  = {} & 1-\frac{2(a+b+c)\ell}{\pi Q}
		+\frac{(\alpha a^2+\beta b^2+\gamma c^2)\ell^2}{2\pi Q^2}
		+\frac{3\ell^2}{2\pi Q}\,,\displaybreak[0]\\[0.15cm]
  p(2) 
  = {} & \frac{2(a+b+c)\ell}{\pi Q}-\frac{(a^2+b^2+c^2)\ell^2}{4Q^2}
		-\frac{(\alpha a^2+\beta b^2+\gamma c^2)\ell^2}{2\pi Q^2}
		-\frac{3\ell^2}{2\pi Q}\,,\displaybreak[0]\\[0.15cm]
  p(3) 
  = {} & \frac{(a^2+b^2+c^2)\ell^2}{2 Q^2} 
	     -\frac{(\alpha a^2+\beta b^2+\gamma c^2)\ell^2}{2\pi Q^2}
	     -\frac{3\ell^2}{2\pi Q}\,,\displaybreak[0]\\[0.15cm]
  p(4) 
  = {} & \frac{3\ell^2}{2\pi Q}-\frac{(a^2+b^2+c^2)\ell^2}{4Q^2} 
		+\frac{(\alpha a^2+\beta b^2+\gamma c^2)\ell^2}{2\pi Q^2}\,,
		\displaybreak[0]\\[0.15cm]
  p(5) = {} & 0 \,,\quad p(6) = 0\,.                      	
\end{align*}
$p(1)$ is the result obtained by Markov (see \eqref{p(1)}). 

For the random throw of a needle of length $\ell\leq\sqrt{3}\,a/2$ onto the lattice $\mathcal{R}_{a,\,a,\,a}$ of equilateral triangles we have $\alpha=\beta=\gamma=\pi/3$, $Q=\sqrt{3}\,a^2/2$, and therefore
\begin{align*}
  p(1) = {} & 1-\frac{4\,\sqrt{3}}{\pi}\,\frac{\ell}{a}
		     +\left(\frac{\sqrt{3}}{\pi}+\frac{2}{3}\right)\frac{\ell^2}{a^2}\,,\quad
  p(2) = \frac{4\,\sqrt{3}}{\pi}\,\frac{\ell}{a}-\left(\frac{\sqrt{3}}{\pi}
		+\frac{5}{3}\right)\frac{\ell^2}{a^2}\,,\\[0.1cm]
  p(3) = {} & \left(\frac{4}{3}-\frac{\sqrt{3}}{\pi}\right)\frac{\ell^2}{a^2}\,,\quad
  p(4) = \left(\frac{\sqrt{3}}{\pi}-\frac{1}{3}\right)\frac{\ell^2}{a^2}\,,\quad
  p(5) = 0\,,\quad p(6) = 0\,.
\end{align*}
This is the result obtained by Santal\'o \cite[pp.\:167-170]{Santalo1} (see also \cite[pp.\:140-141]{Santalo2}).

\subsection{Ellipses}

Let $g$ and $h$ be the lengths of the major axis and the minor axis, respectively. We have $F=\pi gh/4$ and 
\beq
  w(\phi) 
	= \sqrt{g^2\cos^2\phi+h^2\sin^2\phi} = g\,\sqrt{1-\mu^2\sin^2\phi}\,,\quad
  \mu^2 = 1-\left(\frac{h}{g}\right)^2,
\eeq
hence
\beq
  u = \int_0^{\pi}w(\phi)\,\dd\phi = 2\int_0^{\pi/2}w(\phi)\,\dd\phi 
    = 2gE(\mu)\,, 
\eeq
where
\beq
  E(\mu) = \int_0^{\pi/2}\sqrt{1-\mu^2\sin^2\phi}\;\dd\phi
\eeq 
is the complete elliptic integral of the second kind. Furthermore, one has
\beq
  I(x) = g^2\int_0^{\pi}\sqrt{(1-\mu^2\sin^2\phi)(1-\mu^2\sin^2(\phi+x))}\;\dd\phi\,.    
\eeq
Hence, if the inequality condition \eqref{max} with 
\beq
  s(\phi) = \frac{g}{2}\,\sqrt{1-\mu^2\sin^2\phi}
\eeq
is fulfilled, we have found all hitting probabilities for ellipses and $\R$.

For an ellipse with $h=0$ we have $\mu=1$, hence
\begin{align*}
  u
  = {} & 2g\int_0^{\pi/2}\sqrt{1-\sin^2\phi}\;\dd\phi
  = 2g\int_0^{\pi/2}|\cos\phi|\;\dd\phi = 2g\int_0^{\pi/2}\cos\phi\;\dd\phi
  = 2g\,,
\end{align*}
and
\begin{align*}
  I(x) 
  = {} & g^2\int_0^{\pi}\sqrt{(1-\sin^2\phi)(1-\sin^2(\phi+x))}\;\dd\phi\\
  = {} & g^2\int_0^{\pi}\sqrt{\cos^2\phi\,\cos^2(\phi+x)}\;\dd\phi
  = g^2\int_0^{\pi}\left|\cos\phi\cos(\phi+x)\right|\;\dd\phi\displaybreak[0]\\
  = {} & g^2\left(\int_0^{\pi/2-x}-\int_{\pi/2-x}^{\pi/2}+\int_{\pi/2}^\pi\right) 
		\cos\phi\cos(\phi+x)\,\dd x\displaybreak[0]\\
  = {} & g^2\left(\int_0^{\pi/2-x}-\int_{\pi/2-x}^{\pi/2}+\int_{\pi/2}^\pi\right)
		\bigg(\frac{1}{2}\cos x+\frac{1}{2}\cos x\cos 2\phi\\
       &	-\frac{1}{2}\sin x\sin 2\phi\bigg)\,\dd\phi
  = g^2\left(\frac{1}{2}\,(\pi-2x)\cos x+\sin x\right). 
\end{align*}
These formulas are equal to the respective formulas of the rectangle for $h=0$. Therefore, as was to be expected, we get needles also as special cases of ellipses.

\section{Half disc}

Now, we calculate the hitting probabilities in the case that $\C$ is a half disc of radius $r$.

\begin{corollary}
Let $\R$ be a lattice of acute or right triangles. Let $\C$ be a half disc with radius $r$, and satisfying condition \eqref{max}. Then
\begin{align*}
  p(1)
  = {} & 1-\frac{(a+b+c)u}{\pi Q}+\frac{(a^2+b^2+c^2)r^2}{Q^2}
		+\frac{4(ab+bc+ca)r^2}{\pi Q^2}\\
       &	-\frac{(\alpha a^2+\beta b^2+\gamma c^2)r^2}{2\pi Q^2}
		+\frac{(\alpha bc+\beta ca+\gamma ab)r^2}{\pi Q^2}
		-\frac{9r^2}{2\pi Q}\,,\displaybreak[0]\\[0.15cm]
  p(2)
  = {} & \frac{(a+b+c)u}{\pi Q}-\frac{13(a^2+b^2+c^2)r^2}{4Q^2}
		-\frac{(8-\pi)(ab+bc+ca)r^2}{\pi Q^2}\\
       &	+\frac{5(\alpha a^2+\beta b^2+\gamma c^2)r^2}{2\pi Q^2}
		-\frac{3(\alpha bc+\beta ca+\gamma ab)r^2}{\pi Q^2}
		+\frac{33r^2}{2\pi Q}\,,\\[0.15cm]
  p(3)
  = {} & \frac{7(a^2+b^2+c^2)r^2}{2Q^2}
		+\frac{(4-2\pi)(ab+bc+ca)r^2}{\pi Q^2}\\
       &	-\frac{7(\alpha a^2+\beta b^2+\gamma c^2)r^2}{2\pi Q^2}
		+\frac{3(\alpha bc+\beta ca+\gamma ab)r^2}{\pi Q^2}
		-\frac{39r^2}{2\pi Q}\,,\\[0.15cm]
  p(4)
  = {} & \frac{(ab+bc+ca)r^2}{Q^2}-\frac{5(a^2+b^2+c^2)r^2}{4Q^2}
       	+\frac{3(\alpha a^2+\beta b^2+\gamma c^2)r^2}{2\pi Q^2}\\
       & -\frac{(\alpha bc+\beta ca+\gamma ab)r^2}{\pi Q^2}
		+\frac{15r^2}{2\pi Q}-\frac{F}{Q}\,,\\[0.15cm]
  p(5)
  = {} & 0\,,\quad p(6)=\frac{F}{Q}\,,		
\end{align*}
with $u=(\pi+2)r$ and $F=\pi r^2/2$.
\end{corollary}

\begin{proof}
The width of the half disc in the direction $\phi$ is given by
\beq
  w(\phi) = r\left(1+\left|\cos\phi\right|\right),
\eeq
see Fig.\:\ref{w_and_s}.

\begin{figure}[h] 
  \begin{center}
	\includegraphics[scale=0.75]{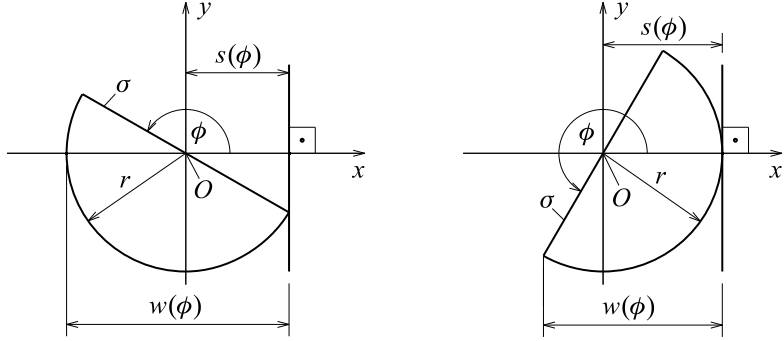}
	\caption{\label{w_and_s} Width function $w$ and support function $s$ for a half disc}
  \end{center}
\end{figure}

\noindent
$w$ has the restriction
\beq
  w(\phi)|_{0\leq\phi<2\pi} = \left\{\ba{@{\:}ccr@{\;}c@{\;}l}
	w_1(\phi) & \mbox{if} & 0\leq & \phi & <\pi/2\,,\\[0.1cm]
	w_{23}(\phi) & \mbox{if} & \pi/2\leq & \phi & <3\pi/2\,,\\[0.1cm]
	w_4(\phi) & \mbox{if} & 3\pi/2\leq & \phi & <2\pi\,,  
  \ea\right.
\eeq
with
\beq
  w_1(\phi) = r(1+\cos\phi)\,,\quad
  w_{23}(\phi) = r(1-\cos\phi)\,,\quad
  w_4(\phi) = r(1+\cos\phi)\,.
\eeq
For the calculation of $I(x)$, $x\in\{\alpha,\beta,\gamma\}$, we have to distinguish the cases
\beq
  0\leq\phi<\pi/2 \,,\quad \pi/2\leq\phi\leq\pi\,,
\eeq
and
\beq
  x\leq\phi+x<\pi/2 \,,\quad \pi/2\leq\phi+x\leq\pi+x\,.
\eeq
Since $0<x\leq\pi/2$, this yields
\beq
  0\leq\phi<\pi/2-x \,,\;\; \pi/2-x\leq\phi<\pi/2 \,,\;\;
  \pi/2\leq\phi<\pi\,;
\eeq
therefore
\begin{align*}
  I(x)
  = {} & \left(\int_0^{\pi/2-x}+\int_{\pi/2-x}^{\pi/2}+
		\int_{\pi/2}^{\pi}\right)w(\phi)w(\phi+x)\:\dd\phi
			\displaybreak[0]\\
  = {} & \int_0^{\pi/2-x}w_1(\phi)\,w_1(\phi+x)\,\dd\phi
		+\int_{\pi/2-x}^{\pi/2}w_1(\phi)\,w_{23}(\phi+x)\,\dd\phi\\
       & +\int_{\pi/2}^{\pi}w_{23}(\phi)\,w_{23}(\phi+x)\,\dd\phi\\
  = {} & (\pi+4)r^2+\frac{1}{2}\,(\pi-2x)r^2\cos x+r^2\sin x\,.
\end{align*}
This, with \eqref{area_Q} and
\begin{align*}
  2bc\cos\alpha = {} & b^2+c^2-a^2,\;\,
  2ac\cos\beta  = a^2+c^2-b^2,\;\, 
  2ab\cos\gamma = a^2+b^2-c^2,
\end{align*} 
yields
\begin{align*}
  & \!\!\!\!\!\!\! bcI(\alpha)+caI(\beta)+abI(\gamma)\\[0.1cm]
  = {} & (\pi+4)(ab+bc+ca)r^2\\
       & +(1/2)bc(\pi-2\alpha)r^2\cos\alpha+bc r^2\sin\alpha\\
       & +(1/2)ca(\pi-2\beta)r^2\cos\beta+ca r^2\sin\beta\\
       & +(1/2)ab(\pi-2\gamma)r^2\cos\gamma+ab r^2\sin\gamma\\[0.1cm]
  = {} & (\pi+4)(ab+bc+ca)r^2\\
       & +(1/4)(\pi-2\alpha)(b^2+c^2-a^2)r^2+Q r^2\\
       & +(1/4)(\pi-2\beta)(c^2+a^2-b^2)r^2+Q r^2\\
       & +(1/4)(\pi-2\gamma)(a^2+b^2-c^2)r^2+Q r^2
			\displaybreak[0]\\[0.1cm]
  = {} & (\pi+4)(ab+bc+ca)r^2+(\pi/4)(a^2+b^2+c^2)r^2+3Qr^2\\
       & -(\alpha/2)(a^2+b^2+c^2-2a^2)r^2
		-(\beta/2)(a^2+b^2+c^2-2b^2)r^2\\
       & -(\gamma/2)(a^2+b^2+c^2-2c^2)r^2
			\displaybreak[0]\\[0.1cm]
  = {} & (\pi+4)(ab+bc+ca)r^2+(\pi/4)(a^2+b^2+c^2)r^2+3Qr^2\\
       & -(\pi/2)(a^2+b^2+c^2)r^2+(\alpha a^2+\beta b^2+\gamma c^2)r^2
			\displaybreak[0]\\[0.1cm]
  = {} & (\pi+4)(ab+bc+ca)r^2-(\pi/4)(a^2+b^2+c^2)r^2\\
       & +(\alpha a^2+\beta b^2+\gamma c^2)r^2+3Qr^2\,.
\end{align*}
The support function $s$ has the restriction
\beq
  s(\phi)|_{0\leq\phi<3\pi} = \left\{\ba{@{\:}ccr@{\;}c@{\;}l}
	s_1(\phi) & \mbox{if} & 0\leq & \phi & <\pi/2\,,\\[0.1cm]
	s_2(\phi) & \mbox{if} & \pi/2\leq & \phi & <\pi\,,\\[0.1cm]
	s_{34}(\phi) & \mbox{if} & \pi\leq & \phi & <2\pi\,,\\[0.1cm]
	s_1(\phi) & \mbox{if} & 2\pi\leq & \phi & <5\pi/2\,,\\[0.1cm] 
	s_2(\phi) & \mbox{if} & 5\pi/2\leq & \phi & <3\pi\,,
  \ea\right.
\eeq
with
\beq
  s_1(\phi) = r\cos\phi\,,\quad
  s_2(\phi) = -r\cos\phi\,,\quad
  s_{34}(\phi) = r\,.
\eeq
For the calculation of $J(x)$, $x\in\{0,\alpha,\beta,\gamma\}$, we have to distinguish the cases
\beq
  0\leq\phi<\pi/2 \,,\quad \pi/2\leq\phi<\pi \,,\quad
  \pi\leq\phi\leq 2\pi
\eeq
and
\beq
  \ba{l}
	x\leq\phi+x<\pi/2 \,,\quad \pi/2\leq\phi+x<\pi \,,\quad
	\pi\leq\phi+x<2\pi\,,\\[0.1cm]
	2\pi\leq\phi+x\leq 2\pi+x\,.
  \ea
\eeq
Since $0\leq x\leq\pi/2$, this yields
\beq
  \ba{l}
	0\leq\phi<\pi/2-x \,,\;\; \pi/2-x\leq\phi<\pi/2 \,,\;\;
	\pi/2\leq\phi<\pi-x\,,\\[0.1cm]
	\pi-x\leq\phi<\pi \,,\;\; \pi\leq\phi<2\pi-x \,,\;\;
	2\pi-x\leq\phi\leq2\pi\,;
  \ea
\eeq
therefore
\begin{align*}
  J(x)
  = {} & \bigg(\int_0^{\pi/2-x}+\int_{\pi/2-x}^{\pi/2}
		+\int_{\pi/2}^{\pi-x}+\int_{\pi-x}^\pi\\
       &	+\int_\pi^{2\pi-x}+\int_{2\pi-x}^{2\pi}\bigg)\,
		s(\phi)s(\phi+x)\:\dd\phi\displaybreak[0]\\
  = {} & \int_0^{\pi/2-x}s_1(\phi)\,s_1(\phi+x)\,\dd\phi
		+\int_{\pi/2-x}^{\pi/2}s_1(\phi)\,s_2(\phi+x)\,\dd\phi
			\displaybreak[0]\\
       & +\int_{\pi/2}^{\pi-x}s_2(\phi)\,s_2(\phi+x)\,\dd\phi
		+\int_{\pi-x}^{\pi}s_2(\phi)\,s_{34}(\phi+x)\,\dd\phi
			\displaybreak[0]\\
       & +\int_{\pi}^{2\pi-x}s_{34}(\phi)\,s_{34}(\phi+x)\,\dd\phi
		+\int_{2\pi-x}^{2\pi}s_{34}(\phi)\,s_{1}(\phi+x)\,\dd\phi
			\displaybreak[0]\\
  = {} & (\pi-x)r^2+\frac{1}{2}\,(\pi-3x)r^2\cos x
		+\frac{5}{2}\,r^2\sin x\,.
\end{align*}
This, with \eqref{area_Q} and
\begin{align*}
  2bc\cos\alpha = {} & b^2+c^2-a^2,\;\,
  2ac\cos\beta  = a^2+c^2-b^2,\;\, 
  2ab\cos\gamma = a^2+b^2-c^2,
\end{align*}
gives
\begin{align*}
  & \!\!\!\!\!\!\! bcJ(\alpha)+caJ(\beta)+abJ(\gamma)\\[0.1cm]
  = {} & bc(\pi-\alpha)r^2+(1/2)bc(\pi-3\alpha)r^2\cos\alpha
		+(5/2)bcr^2\sin\alpha\\
       & + ca(\pi-\beta)r^2+(1/2)ca(\pi-3\beta)r^2\cos\beta
		+(5/2)car^2\sin\beta\\
       & + ab(\pi-\gamma)r^2+(1/2)ab(\pi-3\gamma)r^2\cos\gamma
		+(5/2)abr^2\sin\gamma\\[0.1cm]
  = {} & \pi(ab+bc+ca)r^2-(\alpha bc+\beta ca+\gamma ab)r^2\\
       & +(1/4)(\pi-3\alpha)(b^2+c^2-a^2)r^2
		+(1/4)(\pi-3\beta)(c^2+a^2-b^2)r^2\\
       & +(1/4)(\pi-3\gamma)(a^2+b^2-c^2)r^2+(15/2)Qr^2
			\displaybreak[0]\\[0.1cm]
  = {} & \pi(ab+bc+ca)r^2-(\alpha bc+\beta ca+\gamma ab)r^2
		+(\pi/4)(a^2+b^2+c^2)r^2\\
       & -(3\alpha/4)(a^2+b^2+c^2-2a^2)r^2
		-(3\beta/4)(a^2+b^2+c^2-2b^2)r^2\\
       & -(3\gamma/4)(a^2+b^2+c^2-2c^2)r^2+(15/2)Qr^2
			\displaybreak[0]\\[0.1cm]
  = {} & \pi(ab+bc+ca)r^2-(\alpha bc+\beta ca+\gamma ab)r^2
		+(\pi/4)(a^2+b^2+c^2)r^2\\
       & -(3\pi/4)(a^2+b^2+c^2)r^2
		+(3/2)(\alpha a^2+\beta b^2+\gamma c^2)r^2+(15/2)Qr^2\\[0.1cm]
  = {} & \pi(ab+bc+ca)r^2-(\alpha bc+\beta ca+\gamma ab)r^2
		-(\pi/2)(a^2+b^2+c^2)r^2\\
       & +(3/2)(\alpha a^2+\beta b^2+\gamma c^2)r^2+(15/2)Qr^2\,.
\end{align*}
We have $J(0)=3\pi r^2/2$. 
The assertation follows by inserting the formulas for $bcI(\alpha)+caI(\beta)+abI(\gamma)$, $bcJ(\alpha)+caJ(\beta)+abJ(\gamma)$, and $J(0)$ in the formulas for $p(i)$ at the end of the proof of Theorem~\ref{Thm1}.
\end{proof}

\section{Other convex bodies}

Knowing the width function $w$ and the support function $s$ of any plane convex body $\C$, it is possible to calculate the hitting probabilities for $\C$ and $\R$ with Theorem \ref{Thm1} and, if $w(\phi)=2s(\phi)$ for every $\phi\in[0,2\pi)$, with Corollary \ref{w=2s}, respectively. $w$ and $s$ for a parallelogram and a triangle are to be found in \cite[pp.\:319/320]{Ren_Zhang}.

In particular, numerical values for $p(i)$ are easily obtained by numerical integration of the integrals $I(\alpha)$, $I(\beta)$, $I(\gamma)$, $J(0)$, $J(\alpha)$, $J(\beta)$, and $J(\gamma)$. 

\bigskip
\begin{center}
Uwe B\"asel\\[0.15cm]
HTWK Leipzig,\\
Fakult\"at f\"ur Maschinenbau und Energietechnik,\\ 
PF 30 11 66, 04251 Leipzig, Germany\\[0.15cm]
{\small uwe.baesel@htwk-leipzig.de}
\end{center}

\end{document}